\documentclass[a4paper, 11pt, onecolumn]{article}
\usepackage[margin=2.5cm]{geometry}

\usepackage{amsmath}
\usepackage{amsthm}
\usepackage{indentfirst}
\usepackage{bm}
\usepackage{amsfonts,amssymb}
\usepackage{graphicx}
\usepackage{mathrsfs}
\usepackage{subfigure}
\usepackage{graphicx}
\usepackage{multirow}
\usepackage{float}
\usepackage{algorithmic}
\usepackage{algorithm}
\usepackage{hyperref}
\usepackage{paralist}
\usepackage{xcolor}
\usepackage{pstricks}
\usepackage{lineno}
\usepackage{rotating}
\usepackage{overpic}
\usepackage{lineno}

\newtheorem{theorem}{Theorem}
\newtheorem{lemma}{Lemma}
\newtheorem{remark}{Remark}
\newtheorem{proposition}{Proposition}
\newtheorem{definition}{Definition}

\DeclareMathOperator{\Div}{div}
\renewcommand{\div}{\Div}
\newcommand{\RR}{\mathbb{R}}

\DeclareMathOperator*{\argmin}{arg\,min}
\DeclareMathOperator*{\argmax}{arg\,max}

\newcommand{\radon}{\mathcal{M}}

\newcommand{\wrightarrow}{\rightharpoonup}

\begin{document}

\title{A Wasserstein distance and total variation regularized model to image reconstruction problems
}

\author{
	\normalsize
	Yiming Gao\footnotemark[1] }

\renewcommand{\thefootnote}{\fnsymbol{footnote}}

\footnotetext[1]{School of Mathematics,  Nanjing University of Aeronautics and Astronautics, Nanjing 210016, China, Email:gaoyiming@nuaa.edu.cn}

\date{}
\maketitle

\begin{abstract}
	Optimal transport has gained much attention in image processing field, such as computer vision, image interpolation and medical image registration. Recently, Bredies et al. (ESAIM: M2AN 54:2351–2382, 2020) and Schmitzer et al. (IEEE T MED IMAGING 39:1626-1635, 2019) established the framework of optimal transport regularization for dynamic inverse problems. In this paper, we incorporate Wasserstein distance, together with total variation, into static inverse problems as a prior regularization. The Wasserstein distance formulated by Benamou-Brenier energy measures the similarity between the given template and the reconstructed image. Also, we analyze the existence of solutions of such variational problem in Radon measure space. Moreover, the first-order primal-dual algorithm is constructed for solving this general imaging problem in a specific grid strategy. Finally, numerical experiments for undersampled MRI reconstruction are presented which show that our proposed model can recover images well with high quality and structure preservation.
\end{abstract}



\section{Introduction}
The aim of this paper is to find a Radon measure $\nu\in\mathcal{M}(\overline{\Omega})$ on a closed domain $\overline{\Omega}\subset \RR^d$, such that
\begin{equation}\label{ip}
K\nu = f,
\end{equation}
where $f$, belonging to some Hilbert space $H(\overline{\Omega})$, is some given (noisy) data, and $K:\mathcal{M}(\overline{\Omega})\to H(\overline{\Omega})$ is a linear and weak$^\ast$-to-weak continuous observation operator. 
In general, the problem \eqref{ip} is ill-posed and always needs to be regularized  as a variational model \cite{engl1996regularization,scherzer2009variational}, reading as 
\begin{equation}\label{reg_model}
\min_{\nu\in \mathcal{M}(\overline{\Omega})} \frac{\alpha}{2}\|K\nu-f\|^2_{H}+R(\nu),
\end{equation}
where $\alpha>0$ is a parameter. Note that $R$ is the regularization term restricting the characteristic of solutions. For mathematical imaging problems, the design of regularizers has gained much attention in the last decades, which is a well-established technique, such as, for instance, total variation (TV) \cite{ROF}, total generalized variation (TGV) \cite{TGV}, oscillation TGV \cite{gao2018infimal} and framelet-based regularizer \cite{framelet,cai2012image} and so on. In this paper, we incorporate some structure priors, encoded by a given measure $\mu$, for $\nu$ though optimal transport theory. In particular, the regularizer $R$ is defined by
\begin{equation}
R(\nu):=\beta\|\nabla \nu\|_{\mathcal{M}}+W_2(\mu,\nu),
\end{equation}
with $\beta>0$. The first term appearing in $R$ is the standard TV, which enforces the regularity on $\nu$. The second term $W_2$ is the 2-Wasserstein distance between $\nu$ and the given prior $\mu$. This forces $\nu$ to inherit some structural features from $\mu$.  In general, the Wasserstein distance $W_2$, as shown in \cite{benamou2000computational}, can be equivalently formulated as 
\begin{equation}\label{wasser2}
W_2(\mu,\nu)=\min_{\rho\in \mathcal{M}([0,1]\times\overline{\Omega}),m\in \mathcal{M}([0,1]\times\overline{\Omega};\RR^d)} B_{\mu,\nu}(\rho,m),
\end{equation}
where
\begin{equation*}
B_{\mu,\nu}(\rho,m):=\frac{1}{2}\int_0^1\int_{\overline{\Omega}} \Big|\frac{dm}{d\rho}\Big|^2d\rho, \quad \text{s.t. $\partial_t\rho+\div m=0$, $\rho_0=\mu$, $\rho_1=\nu$}.
\end{equation*}
Note that the above formula employs the fact that $\rho$ and $m$ can be disintegrated into $\rho=dt\otimes \rho_t$, $m=dt\otimes m_t$. This formulation is the classical dynamic optimal transport, called Benamou-Brenier energy. Accordingly, our minimization problem \eqref{reg_model} is equivalent to 
\begin{equation}\label{proMode}
\begin{aligned}
\min_{\rho\in \mathcal{M}([0,1]\times\overline{\Omega}),m\in \mathcal{M}([0,1]\times\overline{\Omega};\RR^d)} &\frac{\alpha}{2}\|K\rho_1-f\|^2_{H}+ \beta\|\nabla \rho_1\|_{\mathcal{M}}+B_{\mu,\rho_1}(\rho,m),\\
&\text{s.t. $\partial_t\rho+\div m=0$, $\rho_0=\mu$.}
\end{aligned}
\end{equation}
Note that the measure $\mu$ (the initial state) is given, which can be regarded as a template for our reconstructed measure $\rho_1$ and can provide some prior information. The Benamou-Brenier energy term $B_{\mu,\rho_1}$ is expected to minimize the Wasserstein  distance between this given template and the reconstruction. In medical imaging, ones can generate an atlas from a specific organ dataset, e.g., brain images, then this atlas can be regarded as a template. Moreover, for dynamic cardiac MRI reconstruction, there exists elastic deformation between two adjacent frames due to the heart beating, therefore, the former frame image can service as a template for the latter one.

In recent years, template based models have gained much attention in inverse problems, that can incorporate prior knowledge into regularization functional \cite{oektem2017shape,chen2019new,karlsson2017generalized,benamou2015iterative}. If a template image is known to be close in some sense to the reconstruction, we can define a regularity term to measure the closedness of them. In \cite{karlsson2017generalized,benamou2015iterative}, the authors employed the entropic Wasserstein distance based on Kantorovich formulation to penalize the template prior information. In \cite{oektem2017shape,chen2018indirect,gris2019incorporation}, Chen et al. introduced large deformation diffeomorphic metric mapping (LDDMM) framework to model large deformations between the template and the target by diffeomorphism flows, which is a shape-regularized method with topology preserving. Later, Gris et al. generalized this framework by the metamorphosis in \cite{gris2020image}, which allows greyscale value change except for the diffeomorphic deformation. Note that the above LDDMM based methods can preserve the topology by restricting the velocity in an admissible Hilbert space.  Our proposed optimal transport based model \eqref{proMode} does not possess this topology-preserving property, while holding the mass-conservation due to the continuity equation. 

Returning to optimal transport, we begin to review some literatures, which has gained much interest in computational and applicable fields, such as machine learning \cite{WGAN2017improved,Wasserstein2017} inverse problem \cite{bredies2020optimal,Bredies2021,karlsson2017generalized},  image processing \cite{ni2009local,papadakis2015multi,maas2015generalized,maas2016generalized,metivier2016optimal,fitschen2015dynamic}, computer vision \cite{rubner2000earth,sadeghian2015automatic,haker2004optimal} and others. The original optimal transport problem was proposed by Monge in 1781 \cite{monge1781memoire}. Given two probability measures $\mu\in\mathcal{P}(X)$ and $\nu\in\mathcal{P}(Y)$ on metric spaces $X$ and $Y$, respectively, of the same mass, one seeks the transport map $T: X\to Y$, such that
\begin{equation}\label{Monge}
\inf_T  \left\{\int_X C(x,T(x))d\mu: \quad T_\#\mu=\nu\right\},
\end{equation}
where $C: X\times Y\to [0,+\infty]$ is the cost function. This problem exhibits several difficulties, one of which is the non-convexity.  In \cite{Kantorovich}, Kantorovich proposed a relaxed formulation of
\begin{equation}\label{Kantorovich}
\inf_\gamma \left\{ \int_{X\times Y} C(x,y)d\gamma: \gamma\in\Pi(\mu,\nu) \right\},
\end{equation}
where $\Pi(\mu,\nu):=\{\gamma\in\mathcal{P}(X,Y):(\pi_x)_\# \gamma=\mu, (\pi_y)_\# \gamma=\nu\}$ is the set of so-called transport plans. Notice that the Kantorovich problem is convex and well-structured, and in the discrete setting, it becomes a linear programming problem. This can be solved efficiently using Sinkhorn iterations \cite{sinkhorn1964relationship,sinkhorn2013} by incorporating an entropic barrier term. Recently, some other regularized transport approaches based on Kullback-Leibler divergence have been proposed \cite{benamou2015iterative,peyre2015entropic}, that can be computed by iterative Bregman projection and Dykstra's algorithm. The  well-known dynamic optimal transport was proposed in \cite{benamou2000computational}, which, equivalent to 2-Wasserstein distance as mentioned in \eqref{wasser2}. That, and its variants appearing subsequently, have been applied for image processing and computer vision tasks \cite{papadakis2014optimal,chizat2018interpolating,papadakis2015multi}. Also, the development of efficient algorithms to
the dynamic optimal transport has attracted much attention. Papadakis et al. \cite{papadakis2014optimal} designed proximal splitting methods including Douglas-Rachford splitting method and primal-dual method to solve the problem. Both of them need a projection onto the divergence-free constraint at each iteration, which amounts to solving a 3D Poisson equation for 2D image. To avoid this projection, \cite{henry2019primal}  introduced the Helmholtz-Hodge decomposition to get rid of this constraint. Moreover, for the Wasserstein-1 distance problem, the multilevel primal-dual algorithm was discussed in \cite{liu2021multilevel}.

Recently, the dynamic optimal transport regularization has been established to solve inverse problems \cite{bredies2020optimal,schmitzer2019dynamic}. In \cite{bredies2020optimal}, Bredies et al. proposed and studied the 
Wasserstein-Fisher-Rao energy for dynamic undersampled MRI reconstruction, and they also showed the well-posedness of the associated regularization model. This is a new functional-analytic framework for recovering curves of Radon measures from continuously acquired measurements, and heavily contributes to the theoretical analysis of our paper, e.g., existence of solutions.  In particular, the extremal points of the unit ball corresponding to optimal transport energies were characterized in \cite{Bredies2021,bredies2022superposition}, and their sparse solutions to the dynamic inverse problems were also analyzed. Meanwhile, the generalizations of conditional gradient method to solve the above optimal transport regularized problems were developed in \cite{bredies2022generalized,bredies2023asymptotic,duval2021dynamical}. In this paper, we aim to address the dynamic optimal transport to static inverse problems, and the main contributions of our work are listed as follows.

\begin{itemize}
	\item Firstly, we establish a variational model based on Benamou-Brenier energy (Wasserstein distance) and total variation for static inverse problems with a known template. The template for the reconstruction is naturally assigned as the initial state of continuity equation.
	\item Secondly, we analyze the existence of solutions to our variational problem in Radon measure space. Also,  we propose the first-order primal-dual method to solve our problem in discrete setting and further introduce staggered grids strategy for the continuity equation. Numerical experiments show that our model can reconstruct images well under low sampling rate for MRI reconstruction.
\end{itemize}

The rest of this paper is organized as follows. Section 2 provides some preliminaries and some properties of Benamou-Brenier energy. We exhibit the proposed Wasserstein prior and total variation based model in Section 3, and, in particular, the existence of minimizers is shown. Section 4 is devoted to the numerical algorithm for our model and grid discretization. Numerical experiments for undersampled MRI reconstruction are addressed in Section 5. Finally, we draw conclusions in Section 6.

\section{Preliminaries}

Before going to establish our proposed model, we begin with some basic notations about the Radon measure theory, which are mainly from \cite{ambrosio2000functions}. Let $X$ be locally compact and separable metric space and $\mathcal{B}(X)$ be the Borel $\sigma$-algebra on $X$. Denote $\mu:\mathcal{B}(X)\to\RR^d$ by an $\RR^d$-valued measure and its total variation measure  by $|\mu|$. We call $\mu$ a finite Radon measure if $|\mu|(X)<\infty$, and the corresponding space is denoted by $\mathcal{M}(X,\RR^d)$ (and $\mathcal{M}(X)$ if $d=1$). Notice that $(C_0(X,\RR^d))^\ast=\mathcal{M}(X,\RR^d)$,
where
\begin{equation*}
C_0(X,\RR^d)=\overline{\{u\in C(X,\RR^d):\text{supp $u$} \subset\subset X \}}.
\end{equation*}
In addition, $\mathcal{M}(X,\RR^d)$ is a Banach space under the induced Radon norm
\begin{equation*}
\|\mu\|_\mathcal{M}=\sup\left\{\langle\mu,\phi\rangle:\phi\in C_0(X,\RR^d), \|\phi\|_\infty\leq 1 \right\}.
\end{equation*}
We mention that the Radon norm is also written as $|\mu|(X)$ in the literature; see, for instance, \cite[Proposition 1.47]{ambrosio2000functions}. We say that $\{\mu_n\}$ in $\mathcal{M}(X,\RR^d)$ converges weakly$^\ast$ to $\mu$ if 
\begin{equation*}
\sum_{i=1}^d\int_X u_id\mu_i^n\to\sum_{i=1}^d\int_X u_id\mu_i, \quad \forall u\in C_0(X,\RR^d).
\end{equation*}
Notice that the mapping $\mu \mapsto\|\mu\|_\mathcal{M}$ is lower semi-continuous in this weak$^\ast$ sense. Moreover, the space $\mathcal{M}(X,\RR^d)$ is compact, i.e., every bounded sequence in it has a weak$^\ast$ convergent subsequence  \cite{ambrosio2000functions}.

In the following, we let  $\Omega\subset\RR^d$ be
an open and bounded domain, and consider a time variable $t\in T:=[0, 1]$. Then we set $X:=T\times\overline{\Omega}$ to be the time-space cylinder. Let $\radon=\radon(X)\times\radon(X,\RR^d)$ and $\radon^+(X)$ be the nonnegative measure space. To simplify the following description, we define the convex set as
\begin{equation*}\label{E_defi}
E:=\{(a,b)\in\RR\times\RR^d:a+\frac{1}{2}|b|^2\leq0\},
\end{equation*}
and the corresponding indicator function as
\begin{equation*}
\delta_E(a,b) = 
\begin{cases}
0,&\text{if $(a,b)\in E$;}\\
+\infty, & \text{otherwise.}
\end{cases}
\end{equation*}
Also, we introduce the map $\Psi:\RR\times\RR^d \to [0,+\infty]$ of
\begin{equation*}
\Psi(t,x):=\begin{cases}
\frac{|x|^2}{2t}, &\text{if $t>0$};\\
0, & \text{if $t=|x|=0$};\\
+\infty, &\text{otherwise.}
\end{cases}
\end{equation*}
Consequently, $\Psi$ is the Legendre conjugate of $\delta_E$, as shown in the following Lemma. The proof is omitted, and we refer the interested
reader to \cite[Lemma 5.17]{santambrogio2015optimal}.


\begin{lemma}
	Let $(a,b)\in\RR\times\RR^d$. Then for $(t,x)\in\RR\times\RR^d$, we have
	\begin{equation}
	\sup_{(a,b)\in E} (at+b\cdot x) = \Psi(t,x).
	\end{equation}
	In particular, $\Psi$ is convex, lower semicontinuous and 1-homogeneous.
\end{lemma}	
By the property of $\Psi$, we now give the dual representation of the Benamou-Brenier energy, and some properties of the energy $B$ are summarized; see, for instance, \cite{bredies2020optimal,Bredies2021,santambrogio2015optimal} for more details.
\begin{definition}
	Let $(\rho,m)\in\mathcal{M}$. We define 
	\begin{equation}\label{BBdual}
	B(\rho,m):=\sup\left\{\int_Xad\rho+\int_Xb\cdot dm :(a,b)\in C_0(X;E)\right\}.
	\end{equation}
\end{definition} 

\begin{proposition}\label{pro:BB}
	The functional $B$ defined in \eqref{BBdual} is convex and lower semicontinuous for the	weak$^\ast$ convergence. Moreover, it satisfies the following properties: 
	\begin{enumerate}[(i)]
		\item $B(\rho,m)\geq 0$ for all $(\rho,m)\in\mathcal{M}$;
		\item assume that $\rho,m\ll\lambda$, for some $\lambda\in\mathcal{M}^+(X)$. Then
		\begin{equation*}
		B(\rho,m) = \int_X \Psi(\frac{d\rho}{d\lambda},\frac{dm}{d\lambda})d\lambda=\frac{1}{2}\int_X |\frac{d m}{d\rho}|^2d\rho;
		\end{equation*}  
		\item if $B(\rho,m)<+\infty$, then $\rho\geq0$ and $m\ll\rho$;
		\item if $\rho\geq0$ and $m\ll\rho$, then $m=\rho v$ for a measurable map $v:X\to \RR^d$ and 
		\begin{equation*}
		B(\rho,m) = \int_X \Psi(1,v)d\rho=\frac{1}{2}\int_X |v|^2d\rho.
		\end{equation*}
	\end{enumerate}
\end{proposition}


The following lemma shows that the measure $m$ can be controlled by the Benamou-Brenier energy and $\rho$, which plays a key role in Lemma \ref{lemma_compact}.
\begin{lemma}[Integrability estimate \cite{maas2015generalized}]\label{lemma_integra}
	Let $\rho\in\mathcal{M}^+(X)$ and $m\in\mathcal{M}(X,\mathbb{R}^d)$. Then,
	\begin{equation*}
	|m|(X)\leq\Big(B(\rho,m)\Big)^\frac{1}{2}\Big(\rho(X)\Big)^\frac{1}{2}.
	\end{equation*}
\end{lemma}

In this paper, we focus on discussing the distributional sense solutions to the continuity equation, which is illustrated as follows.
\begin{definition}
	$(\rho,m)\in\mathcal{M}$ is the  measure solution to  $\partial_t\rho+\div m=0$ if
	\begin{equation}\label{cesolution}
	\int_X \partial_t\phi d\rho + \int_X \nabla\phi\cdot dm = 0, \quad \forall \phi \in C_c^1((0,1)\times\overline{\Omega}).
	\end{equation}
\end{definition}
Note that the momentum $m$ is equipped with the zero flux boundary condition on $\partial\Omega$. We remark that the above definition is inspired by \cite{Bredies2021,bredies2020optimal,ambrosio2005gradient}. The following proposition is a consequence of the disintegration theorem; see, for instance, \cite[Theorem 5.3.1]{ambrosio2005gradient}.
\begin{proposition}[\cite{bredies2020optimal,Bredies2021}]\label{pro:disin}
	Suppose that $(\rho,m)\in\mathcal{M}$ satisfies \eqref{cesolution}, with $\rho\in\mathcal{M}^+(X)$. Then $\rho$ disintegrates, with respect to $dt$, as $\rho=dt\otimes \rho_t$, with $\rho_t\in\mathcal{M}^+(\overline{\Omega})$, i.e., for every $\varphi(t,x)\in L^1(X,\rho)$
	\begin{equation*}
	\int_X \varphi(t,x)d\rho(t,x) = \int_{0}^1\int_{\overline{\Omega}}\varphi(t,x)d\rho_tdt.
	\end{equation*}
	Moreover, $t\mapsto\rho_t(\overline{\Omega})$ is a constant function with distributional derivative 0, which implies that the total mass  $\rho_t(\overline{\Omega})$ is constant in time.   
\end{proposition}
Furthermore, a curve $t\in[0,1]\mapsto\rho_t\in\mathcal{M}(\overline{\Omega})$ is narrowly continuous if the map
	\begin{equation*}
	t \mapsto \int_{\overline{\Omega}}\varphi(x)d\rho_t(x)
	\end{equation*}
	is continuous for every fixed $\varphi(x)\in C(\overline{\Omega})$. We denote by $C_w([0,1];\mathcal{M}(\overline{\Omega}))$ the set of such curves, and by $C_w([0,1];\mathcal{M}^+(\overline{\Omega}))$ the set of narrowly continuous curves of nonnegative measures. The next proposition shows that the disintegration measures $\rho_t$, $t\in[0,1]$, are narrowly
continuous when $(\rho,m)$ solves the continuity equation. The proof can be referred to \cite[Proposition 2.4]{bredies2020optimal} and \cite[Lemma
8.1.2]{ambrosio2005gradient}.
\begin{proposition}[Continuous representative]\label{pro:conti_repre}
	Let  $(\rho,m)\in\mathcal{M}$ be the solution of \eqref{cesolution} with $\rho\in\mathcal{M}^+(X)$ and $\rho_t\in\mathcal{M}^+(\overline{\Omega})$ be the disintegrate of $\rho$ with respect to $dt$. Assume that $m=dt\otimes v_t\rho_t$ with $v_t:X\to\RR^d$ measurable functions such that
	\begin{equation*}
	\int_0^1\int_{\overline{\Omega}} |v_t(x)|d\rho_t(x)dt<\infty.
	\end{equation*}
	Then there exists a narrowly continuous curve $(t\mapsto\tilde{\rho}_t)\in C_w([0,1];\mathcal{M}^+(\overline{\Omega}))$ such that $\rho_t=\tilde{\rho}_t$ a.e. in $[0,1]$. 
\end{proposition}

\section{The mathematical model}
In this section, we will give our proposed model and the existence of solutions to it. Also, an application tp MRI reconstruction is discussed. Before giving the model, we exhibit some definitions. Let $\mu\in\mathcal{M}^+(\overline{\Omega})$ and define 
\begin{equation}
\mathcal{D}=\{(\rho,m)\in\radon: \partial_t\rho+\div m=0 \quad\text{in the sense of \eqref{cesolution} and $\rho_0=\mu$}\}.
\end{equation}
It is easy to check that the set $\mathcal{D}$ is weak$^\ast$ closed. For $\zeta\in\radon(\overline{\Omega})$, we further define its TV norm associated with the weak derivative as
\begin{equation}
\|\nabla\zeta\|_\radon=\sup\left\{\langle\zeta,\div\phi\rangle:\phi\in C_c^1(\Omega,\RR^d), \|\phi\|_\infty\leq 1\right\}.
\end{equation}
Note that $\zeta\in L^1(\overline{\Omega})$ since $\nabla \zeta$ is a measure.
\begin{lemma}\label{lemma:weakstar}
	$\|\nabla\zeta\|_\radon$ is weak$^\ast$ lower semicontinuous on $\mathcal{M}(\overline{\Omega})$.
\end{lemma}
\begin{proof}
	Let $\zeta^n\overset{\ast}{\wrightarrow} \zeta$ in $\mathcal{M}(\overline{\Omega})$. We can choose a sequence $\phi^k\in C_c^1(\Omega,\RR^d) $  with  $\|\phi^k\|_\infty\leq 1$ such that
	\begin{equation*}
	\|\nabla\zeta\|_\radon = \lim_k \langle\zeta,\div\phi^k\rangle.
	\end{equation*}
	Then by the weak$^\ast$ convergence of $\zeta^n$, we have
	\begin{equation*}
	\begin{aligned}
	\langle\zeta,\div\phi^k\rangle&=\lim_n \langle\zeta^n,\div\phi^k\rangle\\
	&=\liminf_n \langle\zeta^n,\div\phi^k\rangle\\
	&\leq \liminf_n \sup_{\phi\in C_c^1(\Omega,\RR^d),\|\phi\|_\infty\leq 1}\langle\zeta^n,\div\phi\rangle\\
	&= \liminf_n \|\nabla\zeta^n\|_\radon.
	\end{aligned}
	\end{equation*}
	Taking the limit over $k$ yields
	\begin{equation*}
	\|\nabla\zeta\|_\radon\leq \liminf_n \|\nabla\zeta^n\|_\radon.
	\end{equation*}
\end{proof}

Assume that $H(\overline{\Omega})$ is some Hilbert space. Let $\alpha,\beta>0$, $f\in H(\overline{\Omega})$ and $K:\mathcal{M}(\overline{\Omega})\to H(\overline{\Omega})$ is linear and weak$^\ast$-to-weak continuous. We define the functional $J:\mathcal{M}\to [0,+\infty]$:
\begin{equation}\label{J}
J(\rho,m):=B(\rho,m) + \frac{\alpha}{2}\|K\rho_1-f\|_H^2+\beta\|\nabla\rho_1\|_\mathcal{M},
\end{equation}
if $(\rho,m)\in\mathcal{D}$ and $J=+\infty$ otherwise. Then, the proposed variational minimization problem with Wasserstein prior and total variation in this paper reads as 
\begin{equation}\label{inOT}
\min_{(\rho,m)\in\mathcal{D}} J(\rho,m).
\end{equation}
The last two terms in $J$ are common in regularized inverse problems that are called the fidelity term and regularizer. The motivation of adding Wasserstein distance $B(\rho,m)$ is to provide the prior information of the template $\mu$ for the reconstructed $\rho_1$. Indeed,  the Wasserstein prior is a distance functional to characterize the discrepancy between $\mu$ and $\rho_1$. Noting that the variable $(\rho,m)$ in \eqref{inOT} satisfies the continuity equation, which implies that the total mass of the reconstructed image should be equal to the template. Moreover, the topology-preserving property needs not to be satisfied for our framework. The choice of $\mu$ depends on specific applications. In medical imaging, ones can generate an atlas from a specific organ dataset, e.g., brain images, then this atlas can be regarded as a template. Moreover, for dynamic cardiac MRI reconstruction, the structures of adjacent frame images occur deformations due to the heart beating, therefore, the former frame image can be recognized as a template for the latter one.

The following lemma shows the compactness property of the functional $J$ in our proposed model \eqref{inOT} which is important to the existence of solutions. This is obtained referred to the work \cite{bredies2020optimal} with slight changes.

\begin{lemma}\label{lemma_compact}
	Let $f\in H(\overline{\Omega})$, $\mu\in\mathcal{M}^+(\overline{\Omega})$ and $\alpha,\beta>0$. Assume that there exists a constant $C_1>0$ such that the sequence $\{(\rho^n,m^n)\}\in \mathcal{M}$ satisfies 
	\begin{equation}\label{Jbdd}
	\sup_n J(\rho^n,m^n)< C_1.
	\end{equation}
	Then $\rho^n=dt\otimes\rho_t^n$ for some $(t\mapsto\rho_t^n)\in C_w([0,1];\mathcal{M}^+(\overline{\Omega}))$. Furthermore, there exists $(\rho,m)\in\mathcal{D}$ with $\rho=dt\otimes\rho_t$, $(t\mapsto\rho_t)\in C_w([0,1];\mathcal{M}^+(\overline{\Omega}))$ such that, up to subsequence,
	\begin{equation}\label{rho_conver}
	\begin{cases}
	\text{$(\rho^n,m^n)\overset{\ast}{\wrightarrow}(\rho,m)$ weakly$^\ast$ in $\mathcal{M}$,}\\
	\text{$\rho^n_1\overset{\ast}{\wrightarrow} \rho_1$ weakly$^\ast$ in $\mathcal{M}(\overline{\Omega})$}. 
	\end{cases}
	\end{equation}
\end{lemma}
\begin{proof}
	By the boundedness of $J(\rho^n,m^n)$ in \eqref{Jbdd}, we have
	\begin{equation}
	\sup_n B(\rho^n,m^n)< C_1.
	\end{equation}
	Therefore, by Proposition \ref{pro:BB}, it follows that $\rho^n\geq 0$, $m^n=v_t^n\rho^n$ for a measurable map $v_t^n:X\to\RR^d$ such that
	\begin{equation}\label{BBbound}
	\sup_n\int_X |v_t^n(x)|^2d\rho^n(t,x)<+\infty.
	\end{equation}
	By \eqref{Jbdd}, we also have $(\rho^n,m^n)\in\mathcal{D}$.
	 Consequently, it follows from Proposition \ref{pro:disin} that $\rho^n=dt\otimes\rho_t^n$ for some $\rho_t^n\in\mathcal{M}^+(\overline{\Omega}),t\in[0,1]$, and $\rho^n_0=\mu$. In particular, $m^n=dt\otimes(v_t^n\rho_t^n)$. Then by \eqref{BBbound} and Proposition \ref{pro:conti_repre}, we obtain $(t\mapsto\rho_t^n)\in C_w([0,1];\mathcal{M}^+(\overline{\Omega}))$. Moreover, by Proposition \ref{pro:disin}, we have that $\rho_t^n(\overline{\Omega})$ is constant with respect to time $t$, i.e., $\rho^n_t(\overline{\Omega})=\rho^n_0(\overline{\Omega})$, for $t\in(0,1]$. Since $\mu\in\radon^+(\overline{\Omega})$, there exists a constant $C_2>0$ such that
	 \begin{equation}\label{rho_bdd}
	 \|\rho^n\|_\radon=\int_{0}^1\rho_t^n(\overline{\Omega})dt= \rho_0^n(\overline{\Omega})=\mu(\overline{\Omega})<C_2.
	 \end{equation}
	 It is then easy to see that there exists a subsequence (not relabeled) $\rho^n\overset{\ast}{\wrightarrow} \rho$ for some $\rho\in\mathcal{M}(X)$. 
	 
    Also, by Lemma \ref{lemma_integra}, we have 
    \begin{equation*}
    |m^n|(X)\leq \sqrt{B(\rho^n,m^n)\rho^n(X)},
    \end{equation*}
	It follows that
	\begin{equation*}
	\sup_n |m^n|(X)\leq \sup_n\sqrt{B(\rho^n,m^n)\rho^n(X)} <\sqrt{C_1C_2},
	\end{equation*}
	which implies that the sequence measures $\{m^n\}$ has uniformly bounded total variation. Therefore, we can exact a subsequence (not relabeled) $m^n\overset{\ast}{\wrightarrow} m$ for some $m\in\radon(X;\RR^d)$. Due to the $weak^\ast$ closedness of $\mathcal{D}$, we have $(\rho,m)\in \mathcal{D}$. 
	Note that the functional $B$ is $weak^\ast$ lower semicontinuous from Proposition \ref{pro:BB}. Therefore $B(\rho,m)<+\infty$, then we have $\rho=dt\otimes\rho_t$ and $m=dt\otimes(v_t\rho_t)$ with $\rho_t\in C_w([0,1];\mathcal{M}^+(\overline{\Omega}))$ by Proposition \ref{pro:disin}.
	
     Next, we aim to show the second convergence of \eqref{rho_conver}. By $(\rho^n,m^n)$ solving the continuity equation, it follows from Proposition \ref{pro:disin} that the map $t\mapsto\rho_t^n(\overline{\Omega})$ is constant function. Then it is obvious to obtain
    \begin{equation*}
    \|\rho_1^n\|_\radon=\rho_1^n(\overline{\Omega})=\rho_0^n(\overline{\Omega})=\mu(\overline{\Omega})<C_2.
    \end{equation*}
	Consequently, there exists a subsequence (not relabeled) $\rho_1^n\overset{\ast}{\wrightarrow} \rho_1$ for some $\rho_1\in\mathcal{M}(\overline{\Omega})$, which concludes the proof.
	
\end{proof}
 

\begin{theorem}\label{theo:exis}
	Let $f\in H(\overline{\Omega})$ , $\mu\in \mathcal{M}^+(\overline{\Omega})$, $K: \mathcal{M}(\overline{\Omega})\to H(\overline{\Omega})$ be linear and weak$^\ast$-to-weak continuous, and $\alpha,\beta>0$. Then the minimization problem \eqref{inOT} admits a solution.	 
\end{theorem}

\begin{proof}
	Since the energy $J$ is bounded from below, we can seek a minimizing sequence $\{(\rho^n,m^n)\}$ and a constant $C>0$, such that
	\begin{equation}
	\sup_n J(\rho^n,m^n)< C.
	\end{equation}
	Form Lemma \ref{lemma_compact}, we have $\rho^n=dt\otimes\rho_t^n$ and $(t\mapsto\rho_t^n)\in C_w([0,1];\mathcal{M}^+(\overline{\Omega}))$. Also, there exists $(\rho^\ast,m^\ast)\in\mathcal{D}$ with $\rho^\ast=dt\otimes\rho_t^\ast$, $(t\mapsto\rho_t^\ast)\in C_w([0,1];\mathcal{M}^+(\overline{\Omega}))$ such that, up to subsequence, $(\rho^n,m^n)\overset{\ast}{\wrightarrow}(\rho^\ast,m^\ast)$ weakly$^\ast$ in $\radon$ and $\rho_1^n\overset{\ast}{\wrightarrow}\rho_1^\ast$ weakly$^\ast$ in $\radon(\overline{\Omega})$.
    In particular, by the definition of $K$, $K\rho_1^n\wrightarrow K\rho_1^\ast$ weakly in $H(\overline{\Omega})$. Therefore, by the lower semi-continuity of the norm
	with respect to the weak convergence, we have
	\begin{equation*}
	\|K\rho_1^\ast-f\|_H^2\leq  \liminf_{n} \|K\rho_1^n-f\|_H^2.
	\end{equation*}
	Furthermore, by the weak$^\ast$ lower semi-continuity of $B$ and $\|\nabla \rho_1\|_\radon$ in Proposition \ref{pro:BB} and Lemma \ref{lemma:weakstar}, we have
	\begin{equation*}
	J(\rho^\ast,m^\ast)\leq \liminf_n J(\rho^n,m^n),
	\end{equation*}
	which implies that $(\rho^\ast,m^\ast)$ is a minimizer.
	
\end{proof}

\textbf{Application to MRI reconstruction:} We now apply our proposed model to realize the undersampled MRI reconstruction, in which the settings can be referred to \cite{bredies2020optimal}. Let $\Omega\subset\RR^2$ be an open bounded domain representing the image domain and  let $\sigma_1\in\mathcal{M}^+(\RR^2)$ be a measure such that
\begin{itemize}
	\item [(M1)] $\|\sigma_1\|_\mathcal{M}\leq C$, where $C>0$;
	\item [(M2)] the map $t\mapsto \int_{\RR^2}\varphi(x)d\sigma_1(x)$ is measurable for each $\varphi(x)\in C_0(\RR^2;\mathbb{C})$.  
\end{itemize}
Let $L^2_{\sigma_1}(\RR^2;\mathbb{C})$ be a Hilbert space normed by $\|f\|^2_{L^2_{\sigma_1}(\RR^2;\mathbb{C})}:=\int_{\RR^2}|f(x)|^2d\sigma_1(x)$. For a measure $\rho\in\mathcal{M}(\overline{\Omega};\mathbb{C})$, we denote its Fourier transform as
\begin{equation}\label{FourierOpe}
\mathscr{F}\rho(x):=\frac{1}{2\pi}\int_{\RR^2}e^{-i\omega\cdot x}d\rho(\omega),
\end{equation}
 where we extend $\rho$ to be zero outside of $\overline{\Omega}$. Accordingly, we define the linear operator $K:\mathcal{M}(\overline{\Omega};\mathbb{C})\to L^2_{\sigma_1}(\RR^2;\mathbb{C})$ as
 \begin{equation*}
 K \rho:=\mathscr{F}\rho.
 \end{equation*}
Then,  given some $ f\in L^2_{\sigma_1}(\RR^2)$, the corresponding minimization problem with respect to MRI reconstruction is 
\begin{equation}\label{MRImodel}
\min_{(\rho,m)\in\mathcal{D}} B(\rho,m) + \frac{\alpha}{2}\|\mathscr{F}\rho_1-f\|_{L^2_{\sigma_1}(\RR^2;\mathbb{C})}^2+\beta\|\nabla\rho_1\|_\mathcal{M}.
\end{equation}
Note that the sampling pattern for such application is given by $\sigma_1$, and some examples, e.g., continuous sampling and compressed-sensing sampling, can be found in \cite{bredies2020optimal}. Moreover, from \cite[Lemma 5.4]{bredies2020optimal}, the Fourier transform operator in \eqref{FourierOpe} is weak$^\ast$-to-weak continuous under (M1) and (M2), and the minimization problem \eqref{MRImodel} admits a solution according to Theorem \ref{theo:exis}.

\section{Algorithm and numerical grids}
In this section, we firstly employ the primal-dual method \cite{CP} to solve our new variational model based on Wasserstein distance and total variation regularization. Note that the algorithm part is illustrated in some abstract finite-dimensional spaces, and then the numerical grids part follows with a specific discretization. We propose to utilize the centered and staggered grids \cite{anderson1995computational,harlow1965numerical} for $\rho$ and $m$, respectively, in which the later is well-performed in fluid mechanics. 

\subsection{Algorithm}


We emphasize that the algorithm designed in this subsection is devoted to the discrete form of our minimization problem.
Denote that $U$, $V$, $Y$ are some finite-dimensional Hilbert spaces, such as Euclidean space. In particular, $U$ can be disintegrated as $\underbrace{U_x\times,...,\times U_x}_{\text{$N_t$ times}}$ with $U_x$ and $N_t$ representing the space discretization and the number of time discretization, respectively. We further assume that $K:U_x\to Y$ is a linear and continuous operator, and $\rho_1\in U_x$, $f\in Y$. Denote $W=U_x\times U_x$, and the corresponding differential operators are defined as $\nabla: U_x\to W$, $\div:V\to U$, $\partial_t:U\to U$.
We now begin to delineate the algorithm to solve our proposed model based on Wasserstein prior and TV regularization. In order to make the algorithm simple, we do not consider the constraint $\rho_0=\mu$ in the algorithm formulation, but we will force $\rho_0^{l+1}=\mu$ in each iteration, which, in practice, is a common trick. Consequently, the discrete form of our minimization problem \eqref{inOT} without the initial state condition can be written as
\begin{equation}\label{disgenMod}
\begin{aligned}
\min_{(\rho,m)\in U\times V} &\int_{\overline{\Omega}} \int_{0}^{1}\frac{|m|^2}{2\rho}dtdx+\frac{\alpha}{2}\|K\rho_1-f\|^2+\beta \|\nabla \rho_1\|_1 + \delta_{\{0\}}(\partial_t\rho+\div m).\\
\end{aligned}
\end{equation}
For \eqref{disgenMod}, one
can see that Fenchel-Rockafellar duality is applicable, such that primal-dual
solutions are equivalent to solutions of the saddle-point problem
\begin{equation}\label{saddp}
\begin{aligned}
\min_{(\rho,m)\in U\times V} \max_{(\lambda,\eta,\zeta)\in U\times Y\times W}\quad &\mathcal{F}(\rho,m;\lambda,\eta,\zeta):=\int_{\overline{\Omega}}\int_{0}^{1}\frac{|m|^2}{2\rho}dtdx+\langle \lambda, \partial_t\rho+\div m\rangle\\
&+\langle\eta, K\rho_1-f\rangle-\frac{1}{2\alpha}\|\eta\|^2+\langle \zeta, \nabla\rho_1\rangle-\delta_{\{\|\cdot\|_\infty\leq\beta\}}(\zeta),
\end{aligned}
\end{equation}
where $\lambda,\eta,\zeta$ are the dual variables. 

The classical primal-dual algorithm  solves the convex-concave saddle-point problem of the form
\begin{equation}\label{eq:pdorig}
\min_{x\in\mathcal{X}} \ \max_{y\in\mathcal{Y}} \ \langle \mathcal{K}x,y\rangle +G(x)-F^\ast(y),
\end{equation}
where $\mathcal{X},\mathcal{Y}$ are Hilbert spaces,
$\mathcal{K}:\mathcal{X}\rightarrow\mathcal{Y}$ is a continuous linear
mapping, and the functionals
$G:\mathcal{X}\rightarrow {({-\infty,\infty}]}$ and
$F^\ast:\mathcal{Y}\rightarrow {({-\infty,\infty}]}$ are proper, convex and
lower semi-continuous. The problem~\eqref{eq:pdorig} is associated to the
Fenchel--Rockafellar primal-dual problems
\begin{equation}
\min_{x\in\mathcal{X}} \ F(\mathcal{K}x)+G(x), \qquad 
\max_{y \in \mathcal{Y}} \ -F^*(y) -G^*(-\mathcal{K}^*y).
\end{equation}
In order to state the algorithm clearly, we have to give the notion of
resolvent operators $(I+\tau\partial G)^{-1}$ and
$(I+\sigma\partial F^\ast)^{-1}$, respectively, which correspond to
the solution operators of certain minimization problems, the so-called
proximal operators:
\begin{equation*}
\begin{split}
& x^\ast=(I+\tau\partial G)^{-1}(\bar{x})=\argmin_{x\in\mathcal{X}} \ \frac{\|x-\bar{x}\|^2}{2}+\tau G(x),\\
&y^\ast=(I+\sigma\partial F^\ast)^{-1}(\bar{y})=\argmin_{y\in\mathcal{Y}} \ \frac{\|y-\bar{y}\|^2}{2}+\sigma F^\ast(y),
\end{split}
\end{equation*}
where $\tau,\sigma>0$ are step-size parameters we need to choose
suitably. Given the initial point
$(x^0,y^0)\in\mathcal{X}\times\mathcal{Y}$ and set $\bar{x}^0=x^0$,
the primal-dual algorithm of the saddle-point problem of~\eqref{eq:pdorig} can be written as 

\begin{equation}\label{pdgen}
\begin{cases}
&y^{l+1}=(I+\sigma\partial F^\ast)^{-1}(y^l+\sigma \mathcal{K}\bar{x}^l),\\
&x^{l+1}=(I+\tau\partial G)^{-1}(x^l-\tau \mathcal{K}^\ast y^{l+1}),\\
&\bar{x}^{l+1}=2x^{l+1}-x^l.
\end{cases}
\end{equation}

Next, we will delineate the iteration \eqref{pdgen} adapted to our problem \eqref{saddp} which can be
reformulated into the above saddle-point structure by redefining its variables and operators as
\begin{equation*}
x=(\rho,m)\in\mathcal{X}=U\times V,\quad y=(\lambda,\eta,\zeta)\in\mathcal{Y}= U\times Y\times W, \quad 
\mathcal{K}=\begin{pmatrix}
\partial_t & \div\\
(0,...,0,K) & 0 \\
(0,...,0,\nabla) & 0\\
\end{pmatrix},
\end{equation*}
as well as
\begin{equation*}
\begin{split}
& G(x)=\int_{\overline{\Omega}}\int_{0}^{1}\frac{|m|^2}{2\rho}dtdx,\\
& F^\ast(y) = \langle f,\eta\rangle+\frac{1}{2\alpha}\|\eta\|^2+\delta_{\{\|\cdot\|_\infty\leq\beta\}}(\zeta).
\end{split}
\end{equation*}
Then the primal-dual iterations are given by
\begin{equation}\label{pditer}
\left\{
\begin{aligned}
(\lambda^{l+1},\eta^{l+1},\zeta^{l+1})&=\argmax_{\lambda,\eta,\zeta} \mathcal{F}(\bar{\rho}^l,\bar{m}^l;\lambda,\eta,\zeta)-\frac{1}{2\sigma}\|\lambda-\lambda^l\|^2-\frac{1}{2\sigma}\|\eta-\eta^l\|^2-\frac{1}{2\sigma}\|\zeta-\zeta^l\|^2\\
(\rho^{l+1},m^{l+1})&=\argmin_{\rho,m} \mathcal{F}(\rho,m;\lambda^{l+1},\eta^{l+1},\zeta^{l+1})+\frac{1}{2\tau}\|\rho-\rho^l\|^2+\frac{1}{2\tau}\|m-m^l\|^2
\\
(\bar{\rho}^{l+1},\bar{m}^{l+1})& = 2(\rho^{l+1},m^{l+1})-(\rho^l,m^l),
\end{aligned}
\right.
\end{equation}
where $\sigma$ and $\tau$ are the step sizes of dual and primal variables, respectively.

For dual $(\lambda,\eta,\zeta)$ problem, it can be written as
\begin{equation*}
\begin{aligned}
\max_{\lambda,\eta,\zeta} \langle \lambda,\partial_t\bar{\rho}^l+\div \bar{m}^l\rangle-\frac{1}{2\sigma}\|\lambda-\lambda^l\|^2&+\langle\eta,K\bar{\rho_1}^l-f\rangle-\frac{1}{2\alpha}\|\eta\|^2-\frac{1}{2\sigma}\|\eta-\eta^l\|^2\\
&+\langle \zeta, \nabla\bar{\rho_1}^l\rangle-\delta_{\{\|\cdot\|_\infty\leq\beta\}}(\zeta)-\frac{1}{2\sigma}\|\zeta-\zeta^l\|^2.
\end{aligned}
\end{equation*}
Clearly, this problem is not coupled with respect to the three dual variables that therefore can be computed one by one:
\begin{equation}
\begin{cases}
\lambda^{l+1} &= \lambda^l + \sigma(\partial_t\bar{\rho}^l+\div \bar{m}^l);\\
\eta^{l+1}&=\Big(\eta^l + \sigma(K\bar{\rho_1}^l-f)\Big)/(1+\sigma/\alpha);\\
\zeta^{l+1} & = \mathcal{P}_{\beta}(\zeta^l+\sigma\nabla\bar{\rho_1}^l).
\end{cases}
\end{equation}
The corresponding projection operator $\mathcal{P}$ is given as
\begin{equation*}
t^\ast = \mathcal{P}_{\xi}(\bar{t})=\arg\min_{\|t\|_\infty\leq\xi}\frac{\|t-\bar{t}\|^2}{2}=\frac{\bar{t}}{\max(1,\frac{|\bar{t}|}{\xi})}.
\end{equation*}

For primal $(\rho,m)$ problem, it can be formulated as
\begin{equation}\label{ppro}
\begin{aligned}
\min_{\rho,m} \int_{\overline{\Omega}}\int_{0}^{1}\frac{|m|^2}{2\rho}dtdx+\langle \lambda^{l+1},\partial_t\rho+\div m\rangle&+\langle \eta^{l+1},K\rho_1-f\rangle+\langle \zeta^{l+1}, \nabla\rho_1\rangle\\&+\frac{1}{2\tau}\|\rho-\rho^l\|^2+\frac{1}{2\tau}\|m-m^l\|^2.
\end{aligned}
\end{equation}
Note that $(\rho,m)=(0,0)$ when $\rho=0$ from the definition of the energy $B(\rho,m)$. We therefore focus on the discussion to the case of $\rho>0$. Firstly, we consider the minimization problem \eqref{ppro} when $t\in(0,1)$, which, equivalently, proceeds to minimize the following by getting rid of the formulas with respect to $\rho_1$: 
\begin{equation}\label{ppro2}
\min_{\rho,m} \int_{\overline{\Omega}}\int_{0}^{1}\frac{|m|^2}{2\rho}dtdx+\langle \lambda^{l+1},\partial_t\rho+\div m\rangle+\frac{1}{2\tau}\|\rho-\rho^l\|^2+\frac{1}{2\tau}\|m-m^l\|^2.
\end{equation}
As ones can see, the functional \eqref{ppro2} is smooth and strongly convex. Thus, by the first-order optimality condition, it turns out that
\begin{equation*}
\begin{cases}
\frac{m}{\rho}+\div^\intercal\lambda^{l+1}+\frac{1}{\tau}(m-m^l)=0,\\
-\frac{|m|^2}{2\rho^2}+\partial_t^\intercal\lambda^{l+1}+\frac{1}{\tau}(\rho-\rho^l)=0.
\end{cases}
\end{equation*}
More precisely, we can obtain by reformulation 
\begin{equation*}
\begin{cases}
m=\frac{\rho\tilde{m}}{\rho+\tau},\\
(\rho+\tau)^2(\rho-\tilde{\rho})-\frac{\tau}{2}|\tilde{m}|^2=0,
\end{cases}
\end{equation*}
where $\tilde{m}=m^l-\tau\div^\intercal\lambda^{l+1}$ and $\tilde{\rho}=\rho^l-\tau\partial_t^\intercal\lambda^{l+1}$ for simplicity. Obviously, $\rho$ can be offered by solving a three-order polynomial equation, and then $m$ is exhibited explicitly. We mention that, to simplify, we set $\rho_0^{l+1}=\mu$ ($\mu$ is the template) after each iteration.
\begin{remark}
	When $t=1$, the three-order polynomial equation with respect to $\rho_1$ is written as 
	\begin{equation*}
	(\rho_1+\tau)^2(\rho_1-\tilde{\rho}_1)-\frac{\tau}{2}|\tilde{m}_1|^2=0,
	\end{equation*}
	where $\tilde{\rho}_1=\rho^l_1-\tau\big(\partial_t^\intercal\lambda_1^{l+1}+K^\intercal\eta^{l+1}+\nabla^\intercal\zeta^{l+1}\big)$. 
\end{remark}

Combining all the updating formulas, we are now ready to state the algorithm in Algorithm \ref{algo:pd}.
\begin{algorithm}[H]
	\caption{ Primal-dual method for solving the proposed model \eqref{inOT}.}
	\label{algo:pd}
	\textbf{Initialization:} Choose $\tau,\sigma>0$ and $\rho^{0},m^{0},\lambda^{0},\eta^{0},\zeta^{0}$, set $\bar{\rho}^{0}=\rho^{0},\bar{m}^{0}=m^{0}$;\\
	\textbf{Iterations:} For $l=0,1,2,\cdots$, update    \\
	1. $\lambda^{l+1} = \lambda^l + \sigma(\partial_t\bar{\rho}^l+\div \bar{m}^l)$;\\
    2. $\eta^{l+1}=\Big(\eta^l + \sigma(K\bar{\rho_1}^l-f)\Big)/(1+\sigma/\alpha)$;\\
    3. $\zeta^{l+1}  = \mathcal{P}_{\beta}(\zeta^l+\sigma\nabla\bar{\rho_1}^l)$;
	\\
	4. get $\rho^{l+1}$ by solving $(\rho+\tau)^2(\rho-\tilde{\rho})-\frac{\tau}{2}|\tilde{m}|^2=0$;\\
	5. $m^{l+1}=\frac{\rho^{l+1}\tilde{m}^l}{\rho^{l+1}+\tau}$; \\
	6. $(\bar{\rho}^{l+1},\bar{m}^{l+1}) = 2(\rho^{l+1},m^{l+1})-(\rho^l,m^l)$;\\
	\textbf{Until convergence};\\
	\textbf{Return} $(\rho^{l+1},m^{l+1})$.
\end{algorithm}

\subsection{Numerical grids}
As ones know, centered-grid-only discretization strategy is not effective for the proposed problem due to the presence of the continuity equation, since it maybe result in chessboard oscillations. Therefore, we employ the staggered grids on variable $m$ which can be found in \cite{gangbo2019unnormalized,papadakis2014optimal,papadakis2015multi}, while others are on centered grids. In  this paper, we focus on the case of $d=2$ as the same dimension of imaging problem. The space-time interval is given as $\overline{\Omega}\times T=[0,1]^2\times[0,1]$. Let $\Delta x=\frac{1}{n_x-1},\Delta y=\frac{1}{n_y-1},\Delta t=\frac{1}{n_t-1}$ be the step sizes and $n_x,n_y,n_t$ be the point numbers of the centered grids where the density $\rho$ lies on.
 As a consequence, the discrete space-time centered grids are then defined by
 \begin{equation*}
 \begin{aligned}
 \overline{\Omega}_{i,j}\times T_{k}&=\Big(i\Delta x,j\Delta y, k\Delta t\Big), \quad \text{for $i=0,...,n_x-1$,  $j=0,...,n_y-1$ and $k=0,...,n_t-1$}.
 \end{aligned}
 \end{equation*}
and the staggered grids are
 \begin{equation*}
\begin{aligned}
\overline{\Omega}_{i-\frac{1}{2},j}\times T_{k}&=\Big((i-\frac{1}{2})\Delta x,j\Delta y, k\Delta t\Big), \quad \text{for $i=0,...,n_x$,  $j=0,...,n_y-1$ and $k=0,...,n_t-1$},\\
\overline{\Omega}_{i,j-\frac{1}{2}}\times T_{k}&=\Big(i\Delta x,(j-\frac{1}{2})\Delta y, k\Delta t\Big), \quad \text{for $i=0,...,n_x-1$,  $j=0,...,n_y$ and $k=0,...,n_t-1$}.
\end{aligned}
\end{equation*}
These allow us to  design the grids for the variables, such as $\rho,\lambda$ are along at $\overline{\Omega}_{i,j}\times T_{k}$,  $\eta,\zeta$ are along at $\overline{\Omega}_{i,j}$, and the two components of $m$ (e.g., $m^x,m^y$) are along at $\overline{\Omega}_{i-\frac{1}{2},j}\times T_{k}$ and $\overline{\Omega}_{i,j-\frac{1}{2}}\times T_{k}$, respectively.

We now begin to define the corresponding gradient and divergence operator in Algorithm \ref{algo:pd}. Firstly, the discrete time based derivative is given by

\begin{equation*}
\begin{aligned}
(\partial_t \rho)_{i,j,k}=
\begin{cases}
\frac{1}{\Delta t}(\rho_{i,j,1}-\rho_{i,j,0}), & \text{if $k=0$};\\
\frac{1}{2\Delta t}(\rho_{i,j,k+1}-\rho_{i,j,k-1}), & \text{if $0<k<n_t-1$};\\
\frac{1}{\Delta t}(\rho_{i,j,n_t-1}-\rho_{i,j,n_t-2}), & \text{if $k=n_t-1$};\\
\end{cases}
\end{aligned}
\end{equation*}
and its adjoint operator corresponds to
\begin{equation*}
\begin{aligned}
(\partial_t^\intercal \lambda)_{i,j,k}=
\begin{cases}
\frac{1}{\Delta t}(-\frac{\lambda_{i,j,1}}{2}-\lambda_{i,j,0}), & \text{if $k=0$};\\
\frac{1}{\Delta t}(-\frac{\lambda_{i,j,2}}{2}+\lambda_{i,j,0}), & \text{if $k=1$};\\
\frac{1}{2\Delta t}(-\lambda_{i,j,k+1}+\lambda_{i,j,k-1}), & \text{if $1<k<n_t-2$};\\
\frac{1}{\Delta t}(-\lambda_{i,j,n_t-1}+\frac{\lambda_{i,j,n_t-3}}{2}), & \text{if $k=n_t-2$};\\
\frac{1}{\Delta t}(\lambda_{i,j,n_t-1}+\frac{\lambda_{i,j,n_t-2}}{2}), & \text{if $k=n_t-1$}.\\
\end{cases}
\end{aligned}
\end{equation*}
Meanwhile, the divergence operator  with respect to $x-$, $y-$ directions are denoted as
\begin{equation*}
(\div m)_{i,j,k} = \frac{m^x_{i+\frac{1}{2},j,k}-m^x_{i-\frac{1}{2},j,k}}{\Delta x}+\frac{m^y_{i,j+\frac{1}{2},k}-m^y_{i,j-\frac{1}{2},k}}{\Delta y}, \qquad \text{for $i=0,...,n_x-1$, $j=0,...,n_y-1$}.
\end{equation*}
Note that its adjoint operator is with the gradient form which is vector-valued. Thus, we define it as $\div^\intercal=-\begin{pmatrix}
\partial_x\\
\partial_y
\end{pmatrix}$. It then follows that $\div^\intercal \lambda =  -\begin{pmatrix}
\partial_x\lambda\\
\partial_y\lambda
\end{pmatrix}$
with
\begin{equation*}
(\partial_x\lambda)_{i-\frac{1}{2},j,k} = \begin{cases}
\frac{1}{\Delta x}(\lambda_{i,j,k}-\lambda_{i-1,j,k}), &\text{if $i=1,...,n_x-1$},\\
0, & \text{if $i=0,n_x$};
\end{cases}
\end{equation*}
and
\begin{equation*}
(\partial_y\lambda)_{i,j-\frac{1}{2},k} = \begin{cases}
\frac{1}{\Delta y}(\lambda_{i,j,k}-\lambda_{i,j-1,k}), &\text{if $j=1,...,n_y-1$},\\
0, & \text{if $j=0,n_y$}.
\end{cases}
\end{equation*}
Furthermore, the gradient operator, $\nabla  = \begin{pmatrix}\partial^+_x\\
\partial^+_y
\end{pmatrix}$, with respect to $\rho_1$ is represented as 
\begin{align*}
&(\partial^+_x\rho_1)_{i,j}=\begin{cases}
\frac{1}{\Delta x}({\rho_1}_{i+1,j}-{\rho_1}_{i,j}),&\text{if $1\leq i<n_x$},\\
0,&\text{if $i=n_x$},
\end{cases} \\
&(\partial^+_y{\rho_1})_{i,j}=\begin{cases}
\frac{1}{\Delta y}({\rho_1}_{i,j+1}-{\rho_1}_{i,j}),&\text{if $1\leq j<n_y$},\\
0, &\text{if $j=n_y$},
\end{cases}\\
\end{align*}
and the adjoint operator to $\zeta$ is denoted by $\nabla^\intercal\zeta=\partial_x^-\zeta^x+\partial_y^-\zeta^y$ where
\begin{align*}
&(\partial_x^-\zeta^x)_{i,j}=\begin{cases}
{\zeta^x}_{1,j},&\text{if $i=1$},\\
{\zeta^x}_{i,j}-{\zeta^x}_{i-1,j},&\text{if $1<i<n_x$},\\
-{\zeta^x}_{n_x-1,j},&\text{if $i=n_x$},
\end{cases} \\
&(\partial_y^-{\zeta^y})_{i,j}=\begin{cases}
{\zeta^y}_{i,1},&\text{if $j=1$},\\
{\zeta^y}_{i,j}-{\zeta^y}_{i,j-1},&\text{if $1<j<n_y$},\\
-{\zeta^y}_{i,n_y-1},&\text{if $j=n_y$}.
\end{cases}\\
\end{align*}

With above definitions and notations, we can delineate the main updatings with respect to $\rho$ and $m$ in Algorithm \ref{algo:pd}. For simplicity, the superscript $l$ will be omit. The updating of $\rho$ is given by
\begin{equation*}
\begin{aligned}
\rho_{i,j,k} =\max\{0, \text{root}&\Big(1,2\tau-\tilde{\rho}_{i,j,k},\tau^2-2\tau\tilde{\rho}_{i,j,k},\\
&-\tau^2\tilde{\rho}_{i,j,k}-\frac{\tau}{8}\big((m^x_{i+\frac{1}{2},j,k}+m^x_{i-\frac{1}{2},j,k})^2+(m^y_{i,j+\frac{1}{2},k}+m^y_{i,j-\frac{1}{2},k})^2\big)\Big)\},
\end{aligned}
\end{equation*}
where $\tilde{\rho}_{i,j,k}=\rho_{i,j,k}-\tau(\partial_t^\intercal\lambda)_{i,j,k}$. Moreover, $\text{root(a,b,c,d)}$ represents the largest real root of the equation $ax^3+bx^2+cx+d=0$. Then the updating of $m$ reads as
\begin{equation*}
m^x_{i-\frac{1}{2},j,k}=
\begin{cases}
\frac{\rho_{i,j,k}+\rho_{i-1,j,k}}{\rho_{i,j,k}+\rho_{i-1,j,k}+2\tau}\tilde{m}^x_{i-\frac{1}{2},j,k}, & \text{if $i=1,...,n_x-1$},\\
0, & \text{if $i=0,n_x$,}
\end{cases}
\end{equation*}
and
\begin{equation*}
m^y_{i,j-\frac{1}{2},k}=
\begin{cases}
\frac{\rho_{i,j,k}+\rho_{i,j-1,k}}{\rho_{i,j,k}+\rho_{i,j-1,k}+2\tau}\tilde{m}^y_{i,j-\frac{1}{2},k}, & \text{if $j=1,...,n_y-1$},\\
0, & \text{if $j=0,n_y$,}
\end{cases}
\end{equation*}
where $\tilde{m}^x_{i-\frac{1}{2},j,k}=m^x_{i-\frac{1}{2},j,k}+\tau(\partial_x\lambda)_{i-\frac{1}{2},j,k}$ and  $\tilde{m}^y_{i,j-\frac{1}{2},k}=m^y_{i,j-\frac{1}{2},k}+\tau(\partial_y\lambda)_{i,j-\frac{1}{2},k}$.

\section{Numerical experiments}
The experiments in this paper are devoted to the undersampled MRI reconstruction, which recovers an image from given incomplete Fourier data. Note that, in practice, the forward operator  $K=\mathcal{PF}$ is a composition of a sampling operator $\mathcal{P}$  and the Fourier transform $\mathcal{F}$. An example of the sampling operator is shown in Figure \ref{fig:radial}, where the white line region corresponds to the sampling points in frequency domain. Compressed sensing techniques have been shown to be effective for image reconstruction in MRI with undersampling \cite{lustig2007sparse,block2007undersampled}, i.e., where much less data is available than usually required. Particularly, we incorporate Wasserstein prior into this problem to utilize the template structure. Under the framework of Benamou-Brenier energy, the reconstructed image inherits the same mass from the template, while it needs not to preserve the same topology. In practice, the template is always chosen to suffer from some deformations compared to the reconstructed image.

We compare our results with the Zero-filling method and the classical TV regularized model. The Zero-filling sets the unsampled data points in frequency space as zeroes, and then it proceeds to act inverse Fourier transform. The TV model reads as
\begin{equation*}
\min_u\frac{1}{2}\|K u-f\|^2+\alpha\int_{\overline{\Omega}} |\nabla u|dx,
\end{equation*}
which will be abbreviated as TV for simplicity in the following. Also, our proposed model is abbreviated as Wass-TV. The sampling scheme in this paper we choose is the equispaced radial sampling
pattern; see Figure \ref{fig:radial}, where the white line region corresponds to the sampling points in frequency domain. The testing images are listed in Figure \ref{fig:testimage}. In this paper, the templates are got synthetically by acting some deformations on the ground truths. We mention that all the experiments in this paper are noise-free and the case with noise needs to be further studied.
\begin{figure}[H]
	\centering
	\begin{minipage}{0.25\textwidth}		\includegraphics[width=\textwidth]{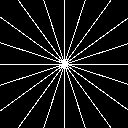}	
    \end{minipage}
\caption{\label{fig:radial}	The sampling pattern.}
\end{figure}

The quality of image recovery is measured by the peak signal to noise ratio (PSNR) and structural similarity (SSIM) that are defined as follows:
\begin{equation*}
\text{PSNR}(u,\underline{u})=10\log_{10}\frac{MN}{\|u-\underline{u}\|^2} \text{dB},
\end{equation*}
where $u$ and $\underline{u}$ are the restored and original images, respectively, and $M,N$ are the channel number and image size, respectively.
$$
\text{SSIM}(u,\underline{u})=\frac{(2\eta_u\eta_{\underline{u}}+c_1)(2\sigma_{u\underline{u}}+c_2)}{(\eta_u^2+\eta_{\underline{u}}^2+c_1)(\sigma_u^2+\sigma_{\underline{u}}^2+c_2)},
$$
where $\eta_u,\eta_{\underline{u}},\sigma_u,\sigma_{\underline{u}},\sigma_{u\underline{u}}$ are the
mean, variance and co-variance of $u$ and $\underline{u}$, and the $c_1$
and $c_2$ are small positive constants. 

To show the stability of our proposed model, we fix the model parameters in \eqref{inOT} by $\alpha=100,\beta=0.001$, and the algorithm parameters in Algorithm \ref{algo:pd} by $\tau=0.001,\sigma=0.01$ (For ``Brain'' image, $\sigma=0.002$). The  time interval is set to $\Delta t=\frac{1}{14}$, i.e., $n_t=15$ for all testings.
All the codes are implemented by MATLAB R2016b running on a desktop with Intel Core i7 CPU at 4.0 GHz and 16 GB of RAM.

\begin{figure}[!htbp]
	\centering	
	\begin{tabular}{ccccc}
		\begin{minipage}{0.01\textwidth}		
		\begin{turn}{90}Ground truth\end{turn}
		\end{minipage}&
		\begin{minipage}{0.2\textwidth}
			\includegraphics[width=\textwidth]{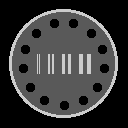}		
		\end{minipage}&
		\begin{minipage}{0.2\textwidth}
			\includegraphics[width=\textwidth]{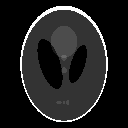}	
		\end{minipage}&
		\begin{minipage}{0.2\textwidth}
			\includegraphics[width=\textwidth]{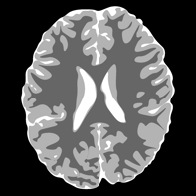}		
		\end{minipage} \\
	    \\[-0.5em]
		\begin{minipage}{0.01\textwidth}		
			\begin{turn}{90}Template ($\rho_0$) \end{turn}
		\end{minipage}&
		\begin{minipage}{0.2\textwidth}
			\includegraphics[width=\textwidth]{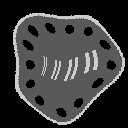}		
		\end{minipage}&
		\begin{minipage}{0.2\textwidth}
			\includegraphics[width=\textwidth]{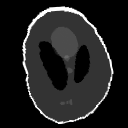}
		\end{minipage}&
    	\begin{minipage}{0.2\textwidth}
		    \includegraphics[width=\textwidth]{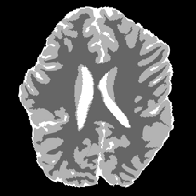}	
	    \end{minipage}\\
       & Phantom & SheppLogan & Brain
	\end{tabular}
	\caption{\label{fig:testimage}	The testing images named ``Phantom'' (128$\times$128), ``SheppLogan'' (128$\times$128) and ``Brain'' (196$\times$196), and their templates. The templates are got synthetically by acting some deformations on the ground truths.}
\end{figure}

\subsection{Experiment 1}

At first, we show the experiments on ``Phantom'' image. As shown in Figure \ref{fig:testimage}, the template ($\rho_0$) is a wrapping image, while the main structure is still preserved that can provide prior information in some sense. Figure \ref{fig:phan} exhibits the restorations of Zero-filling, TV and the proposed Wass-TV method with different sampling rates. The rates correspond to distinct numbers of spokes, e.g., 5, 10 and 15, respectively. Overall, the Zero-filling method suffers from strong artifacts and bad reconstructions, since it does not possess regularization strategy. Looking at the first row, the result of our method seems better than that of TV, where ones can tell some black circles, even it seems very hard. In the third row, the short white lines inside the object are reconstructed well in our model, while TV has some blur artifacts. Consequently, our model incorporated Wasserstein prior can handle this case better, which can further verified by the PSNR and SSIM values listed in Table \ref{table:phan}. 

Furthermore, we choose another image as the template for testing. The new template is exhibited in Figure \ref{fig:nonshape}. As ones can see, this template is distinct from the ground truth with different topology and gray values. The proposed Wass-TV model can also reconstruct this image well compared to Figure \eqref{fig:phan}, which implies that our Wasserstein prior model is not restricted by a topology-preserving template in practice. In general, this pair images are recognized as two modalities in MRI imaging, that inspires us a way to find the template.


\begin{figure}[!htbp]
	\centering	
	\begin{tabular}{cccc}
		\begin{minipage}{0.01\textwidth}		
			\begin{turn}{90} 4.29\%\end{turn}
		\end{minipage}&
		\begin{minipage}{0.2\textwidth}
			\includegraphics[width=\textwidth]{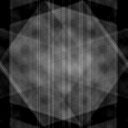}	
		\end{minipage}&
		\begin{minipage}{0.2\textwidth}
			\includegraphics[width=\textwidth]{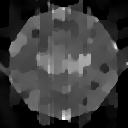}		
		\end{minipage}&
	   \begin{minipage}{0.2\textwidth}
	 	    \includegraphics[width=\textwidth]{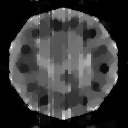}		
	   \end{minipage}\\
   \begin{minipage}{0.1\textwidth}		
   	\begin{overpic}[scale=1]{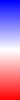}	
   				\put(20,2){\small \color{black}{$\leq -50\%$}}	
   				\put(20,90){\small \color{black}{$\geq 50\%$}}		
   			\end{overpic}
   \end{minipage}&
   \begin{minipage}{0.2\textwidth}
   	\includegraphics[width=\textwidth]{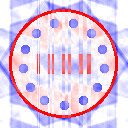}	
   \end{minipage}&
   \begin{minipage}{0.2\textwidth}
   	\includegraphics[width=\textwidth]{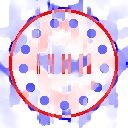}		
   \end{minipage}&
   \begin{minipage}{0.2\textwidth}
   	\includegraphics[width=\textwidth]{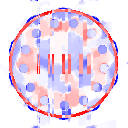}		
   \end{minipage}\\
	\begin{minipage}{0.01\textwidth}		
		\begin{turn}{90}8.48\%\end{turn}
	\end{minipage}	&
		\begin{minipage}{0.2\textwidth}
			\includegraphics[width=\textwidth]{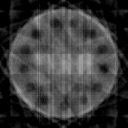}	
		\end{minipage}&
		\begin{minipage}{0.2\textwidth}
			\includegraphics[width=\textwidth]{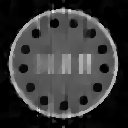}		
		\end{minipage}&
		\begin{minipage}{0.2\textwidth}
			\includegraphics[width=\textwidth]{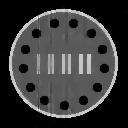}		
		\end{minipage}\\
	\begin{minipage}{0.1\textwidth}		
		\begin{overpic}[scale=1]{bluewhitered.png}	
			\put(20,2){\small \color{black}{$\leq -50\%$}}	
			\put(20,90){\small \color{black}{$\geq 50\%$}}		
		\end{overpic}
	\end{minipage}&
	\begin{minipage}{0.2\textwidth}
		\includegraphics[width=\textwidth]{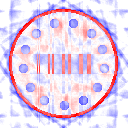}	
	\end{minipage}&
	\begin{minipage}{0.2\textwidth}
		\includegraphics[width=\textwidth]{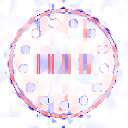}		
	\end{minipage}&
	\begin{minipage}{0.2\textwidth}
		\includegraphics[width=\textwidth]{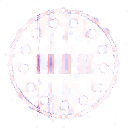}		
	\end{minipage}\\
		\begin{minipage}{0.01\textwidth}		
		\begin{turn}{90}12.68\%\end{turn}
	\end{minipage}	&
	\begin{minipage}{0.2\textwidth}
		\includegraphics[width=\textwidth]{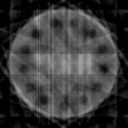}	
	\end{minipage}&
	\begin{minipage}{0.2\textwidth}
		\includegraphics[width=\textwidth]{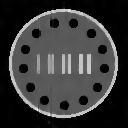}		
	\end{minipage}&
	\begin{minipage}{0.2\textwidth}
		\includegraphics[width=\textwidth]{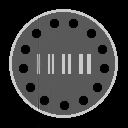}		
	\end{minipage}\\
\begin{minipage}{0.1\textwidth}		
	\begin{overpic}[scale=1]{bluewhitered.png}	
		\put(20,2){\small \color{black}{$\leq -50\%$}}	
		\put(20,90){\small \color{black}{$\geq 50\%$}}		
	\end{overpic}
\end{minipage}&
\begin{minipage}{0.2\textwidth}
	\includegraphics[width=\textwidth]{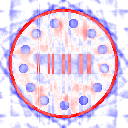}	
\end{minipage}&
\begin{minipage}{0.2\textwidth}
	\includegraphics[width=\textwidth]{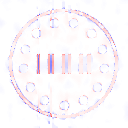}		
\end{minipage}&
\begin{minipage}{0.2\textwidth}
	\includegraphics[width=\textwidth]{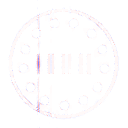}		
\end{minipage}\\
		& Zero-filling & TV & Wass-TV ($\rho_1$)
	\end{tabular}
	\caption{\label{fig:phan}The restorations of Zero-filling, TV and the proposed Wass-TV method on image ``Phantom''. The sampling rates of the first, third and fifth rows are 4.29\%, 8.48\% and 12.68\%, respectively. The even rows exhibits the differences to the ground truth~(blue indicates over and red under estimation).}
\end{figure}

\begin{table}[!htbp]
	\center{}
	\begin{tabular}{|c|ccc|ccc|}
		\hline
		\hline
		\multirow{2}*{sampling rate}&\multicolumn{3}{|c|}{PSNR (dB)} &  \multicolumn{3}{|c|}{SSIM}\\
		\cline{2-7}
		 & Zero-filling & TV & Wass-TV &  Zero-filling & TV & Wass-TV\\
		\hline
		4.29\% (5) & 15.40 & 15.48 & \textbf{17.26} & 0.2497 & 0.3163 & \textbf{0.5344} \\
		8.48\% (10)& 16.90 & 22.02 & \textbf{30.05} & 0.3173 & 0.6918 & \textbf{0.9441} \\
		12.68\% (15)& 18.02 & 28.70 & \textbf{37.60} & 0.3695 & 0.9181 & \textbf{0.9924} \\
		\hline
		\hline	
	\end{tabular}
	\caption{\label{table:phan} PSNR and SSIM values of Zero-filling, TV and our proposed method with different sampling rates on ``Phantom'' image. The numbers in the brackets~(the first column) are the numbers of sampling spokes.}
\end{table}

\begin{figure}[!htbp]
	\centering	
	\begin{tabular}{cccc}
		\begin{minipage}{0.2\textwidth}
			\includegraphics[width=\textwidth]{phan1.png}		
		\end{minipage}&
		\begin{minipage}{0.2\textwidth}
			\includegraphics[width=\textwidth]{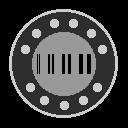}
		\end{minipage}&
		\begin{minipage}{0.2\textwidth}
			\includegraphics[width=\textwidth]{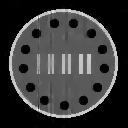}	
		\end{minipage}\\
		Phantom & Template & Wass-TV
	\end{tabular}
	\caption{\label{fig:nonshape}	The ``Phantom'' image reconstruction based on another modality template with 10 spokes (8.48\% sampling rate). Wass-TV: PSNR is 29.94; SSIM is 0.9421.}
\end{figure}

\subsection{Experiment 2}

\begin{figure}[!htbp]
	\centering	
	\begin{tabular}{cccc}
		\begin{minipage}{0.01\textwidth}		
			\begin{turn}{90} 4.29\%\end{turn}
		\end{minipage}&
		\begin{minipage}{0.2\textwidth}
			\includegraphics[width=\textwidth]{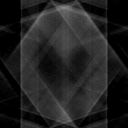}	
		\end{minipage}&
		\begin{minipage}{0.2\textwidth}
			\includegraphics[width=\textwidth]{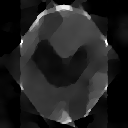}		
		\end{minipage}&
		\begin{minipage}{0.2\textwidth}
			\includegraphics[width=\textwidth]{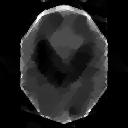}		
		\end{minipage}\\
	 \begin{minipage}{0.1\textwidth}		
		\begin{overpic}[scale=1]{bluewhitered.png}	
			\put(20,2){\small \color{black}{$\leq -50\%$}}	
			\put(20,90){\small \color{black}{$\geq 50\%$}}		
		\end{overpic}
	\end{minipage}&
	\begin{minipage}{0.2\textwidth}
		\includegraphics[width=\textwidth]{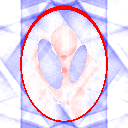}	
	\end{minipage}&
	\begin{minipage}{0.2\textwidth}
		\includegraphics[width=\textwidth]{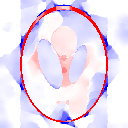}		
	\end{minipage}&
	\begin{minipage}{0.2\textwidth}
		\includegraphics[width=\textwidth]{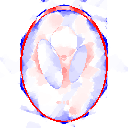}		
	\end{minipage}\\
		\begin{minipage}{0.01\textwidth}		
			\begin{turn}{90}8.48\%\end{turn}
		\end{minipage}	&
		\begin{minipage}{0.2\textwidth}
			\includegraphics[width=\textwidth]{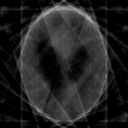}	
		\end{minipage}&
		\begin{minipage}{0.2\textwidth}
			\includegraphics[width=\textwidth]{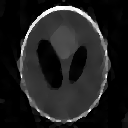}		
		\end{minipage}&
		\begin{minipage}{0.2\textwidth}
			\includegraphics[width=\textwidth]{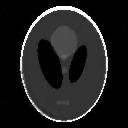}		
		\end{minipage}\\
	 \begin{minipage}{0.1\textwidth}		
		\begin{overpic}[scale=1]{bluewhitered.png}	
			\put(20,2){\small \color{black}{$\leq -50\%$}}	
			\put(20,90){\small \color{black}{$\geq 50\%$}}		
		\end{overpic}
	\end{minipage}&
	\begin{minipage}{0.2\textwidth}
		\includegraphics[width=\textwidth]{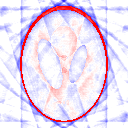}	
	\end{minipage}&
	\begin{minipage}{0.2\textwidth}
		\includegraphics[width=\textwidth]{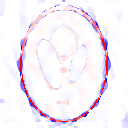}		
	\end{minipage}&
	\begin{minipage}{0.2\textwidth}
		\includegraphics[width=\textwidth]{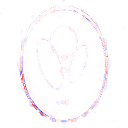}		
	\end{minipage}\\
		\begin{minipage}{0.01\textwidth}		
			\begin{turn}{90}12.68\%\end{turn}
		\end{minipage}	&
		\begin{minipage}{0.2\textwidth}
			\includegraphics[width=\textwidth]{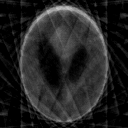}	
		\end{minipage}&
		\begin{minipage}{0.2\textwidth}
			\includegraphics[width=\textwidth]{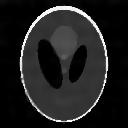}		
		\end{minipage}&
		\begin{minipage}{0.2\textwidth}
			\includegraphics[width=\textwidth]{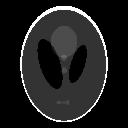}		
		\end{minipage}\\
	 \begin{minipage}{0.1\textwidth}		
		\begin{overpic}[scale=1]{bluewhitered.png}	
			\put(20,2){\small \color{black}{$\leq -50\%$}}	
			\put(20,90){\small \color{black}{$\geq 50\%$}}		
		\end{overpic}
	\end{minipage}&
	\begin{minipage}{0.2\textwidth}
		\includegraphics[width=\textwidth]{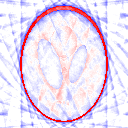}	
	\end{minipage}&
	\begin{minipage}{0.2\textwidth}
		\includegraphics[width=\textwidth]{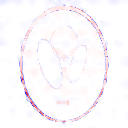}		
	\end{minipage}&
	\begin{minipage}{0.2\textwidth}
		\includegraphics[width=\textwidth]{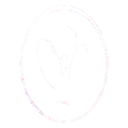}		
	\end{minipage}\\
		& Zero-filling & TV & Wass-TV ($\rho_1$)
	\end{tabular}
	\caption{\label{fig:shepp}The restorations of Zero-filling, TV and the proposed Wass-TV method on image ``SheppLogan''. The sampling rates of the first, second and third rows are 4.29\%, 8.48\% and 12.68\%, respectively. The even rows exhibits the differences to the ground truth~(blue indicates over and red under estimation).}
\end{figure}

Figure \ref{fig:shepp} exhibits the reconstruction results on ``SheppLogan'' image. As one can see, under sampling rates 8.48\% and 12.68\%, although the TV model can recover this image, it, in some sense, loses main objects, such as the edges and some small particles inside the image. These shortcomings can be further verified by the Blue-White-Red error maps where the edge errors of TV are obvious. In contrast, compared to the ground truth, images constructed by our optimal transport based approach seem better with less artifacts, and the above characteristics, such as edges and small particles on the bottom, can be preserved. These imply that ones can improve the reconstruction by incorporating prior information. In addition, the PSNR and SSIM are exhibited in Table \ref{table:shepp}. The values of the proposed Wass-TV method outperform TV, especially for the SSIM values in lower sampling rate.

\begin{table}[!htbp]
	\center{}
	\begin{tabular}{|c|ccc|ccc|}
		\hline
		\hline
		\multirow{2}*{sampling rate}&\multicolumn{3}{|c|}{PSNR (dB)} &  \multicolumn{3}{|c|}{SSIM}\\
		\cline{2-7}
		& Zero-filling & TV & Wass-TV &  Zero-filling & TV & Wass-TV\\
		\hline
		4.29\% (5) & 15.40 & 15.68 & \textbf{17.55} & 0.2989 & 0.3636 & \textbf{0.6218} \\
		8.48\% (10)& 16.61 & 20.80 & \textbf{30.74} & 0.3158 & 0.6927 & \textbf{0.9748} \\
		12.68\% (15)& 17.32 & 28.67 & \textbf{43.54} & 0.3235 & 0.8658 & \textbf{0.9979} \\
		\hline
		\hline	
	\end{tabular}
	\caption{\label{table:shepp} PSNR and SSIM values of Zero-filling, TV and our proposed method with different sampling rates on ``SheppLogan'' image. The numbers in the brackets~(the first column) are the numbers of sampling spokes.}
\end{table}

\subsection{Experiment 3}
Finally, we discuss the experiments on the ``Brain'' image. From the visual effect, the reconstructions of TV are very similar to that of our method, see Figure \ref{fig:brain}, whose structures are recovered well. Particularly, comparing the images carefully at the third and fifth rows, one can observe that the Wass-TV method can smear some artifacts and sharpen the edge structures. This can be verified by the error maps where the corresponding values of our proposed model tend to zero (white color) in most of regions. The PSNR and SSIM exhibited in Table \ref{table:brain} support the observation. 
\begin{figure}[!htbp]
	\centering	
	\begin{tabular}{cccc}
		\begin{minipage}{0.01\textwidth}		
			\begin{turn}{90} 5.58\%\end{turn}
		\end{minipage}&
		\begin{minipage}{0.2\textwidth}
			\includegraphics[width=\textwidth]{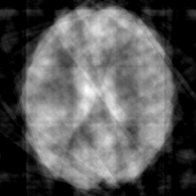}	
		\end{minipage}&
		\begin{minipage}{0.2\textwidth}
			\includegraphics[width=\textwidth]{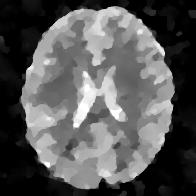}		
		\end{minipage}&
		\begin{minipage}{0.2\textwidth}
			\includegraphics[width=\textwidth]{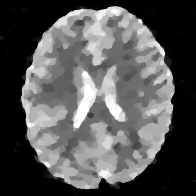}		
		\end{minipage}\\
	\begin{minipage}{0.1\textwidth}		
		\begin{overpic}[scale=1]{bluewhitered.png}	
			\put(20,2){\small \color{black}{$\leq -50\%$}}	
			\put(20,90){\small \color{black}{$\geq 50\%$}}		
		\end{overpic}
	\end{minipage}&
	\begin{minipage}{0.2\textwidth}
		\includegraphics[width=\textwidth]{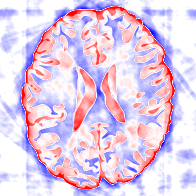}	
	\end{minipage}&
	\begin{minipage}{0.2\textwidth}
		\includegraphics[width=\textwidth]{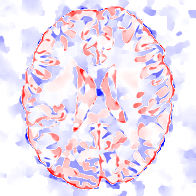}		
	\end{minipage}&
	\begin{minipage}{0.2\textwidth}
		\includegraphics[width=\textwidth]{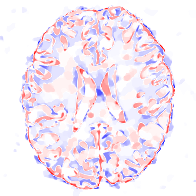}		
	\end{minipage}\\
		\begin{minipage}{0.01\textwidth}		
			\begin{turn}{90}11.40\%\end{turn}
		\end{minipage}	&
		\begin{minipage}{0.2\textwidth}
			\includegraphics[width=\textwidth]{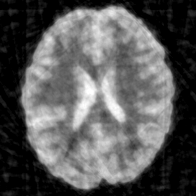}	
		\end{minipage}&
		\begin{minipage}{0.2\textwidth}
			\includegraphics[width=\textwidth]{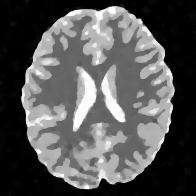}		
		\end{minipage}&
		\begin{minipage}{0.2\textwidth}
			\includegraphics[width=\textwidth]{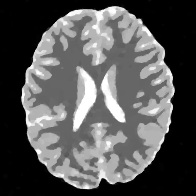}		
		\end{minipage}\\
	\begin{minipage}{0.1\textwidth}		
		\begin{overpic}[scale=1]{bluewhitered.png}	
			\put(20,2){\small \color{black}{$\leq -50\%$}}	
			\put(20,90){\small \color{black}{$\geq 50\%$}}		
		\end{overpic}
	\end{minipage}&
	\begin{minipage}{0.2\textwidth}
		\includegraphics[width=\textwidth]{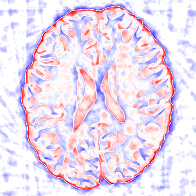}	
	\end{minipage}&
	\begin{minipage}{0.2\textwidth}
		\includegraphics[width=\textwidth]{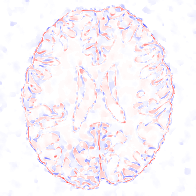}		
	\end{minipage}&
	\begin{minipage}{0.2\textwidth}
		\includegraphics[width=\textwidth]{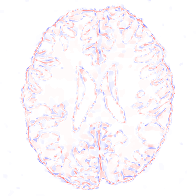}		
	\end{minipage}\\
		\begin{minipage}{0.01\textwidth}		
			\begin{turn}{90}16.37\%\end{turn}
		\end{minipage}	&
		\begin{minipage}{0.2\textwidth}
			\includegraphics[width=\textwidth]{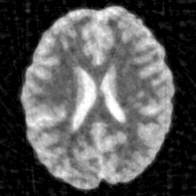}	
		\end{minipage}&
		\begin{minipage}{0.2\textwidth}
			\includegraphics[width=\textwidth]{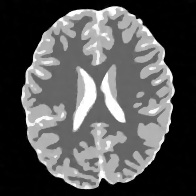}		
		\end{minipage}&
		\begin{minipage}{0.2\textwidth}
			\includegraphics[width=\textwidth]{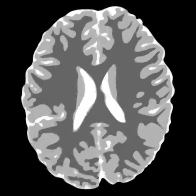}		
		\end{minipage}\\
	\begin{minipage}{0.1\textwidth}		
		\begin{overpic}[scale=1]{bluewhitered.png}	
			\put(20,2){\small \color{black}{$\leq -50\%$}}	
			\put(20,90){\small \color{black}{$\geq 50\%$}}		
		\end{overpic}
	\end{minipage}&
	\begin{minipage}{0.2\textwidth}
		\includegraphics[width=\textwidth]{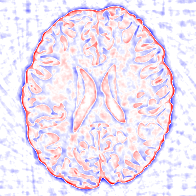}	
	\end{minipage}&
	\begin{minipage}{0.2\textwidth}
		\includegraphics[width=\textwidth]{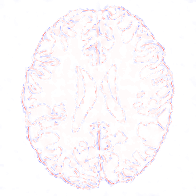}		
	\end{minipage}&
	\begin{minipage}{0.2\textwidth}
		\includegraphics[width=\textwidth]{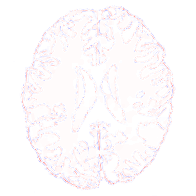}		
	\end{minipage}\\
		& Zero-filling & TV & Wass-TV ($\rho_1$)
	\end{tabular}
	\caption{\label{fig:brain}The restorations of Zero-filling, TV and the proposed Wass-TV method on image ``Brain''. The sampling rates of the first, second and third rows are 5.58\%, 11.40\% and 16.37\%, respectively. The even rows exhibits the differences to the ground truth~(blue indicates over and red under estimation).}
\end{figure}

\begin{table}[!htbp]
	\center{}
	\begin{tabular}{|c|ccc|ccc|}
		\hline
		\hline
		\multirow{2}*{sampling rate}&\multicolumn{3}{|c|}{PSNR (dB)} &  \multicolumn{3}{|c|}{SSIM}\\
		\cline{2-7}
		& Zero-filling & TV & Wass-TV &  Zero-filling & TV & Wass-TV\\
		\hline
		5.58\% (10) & 16.67 & 20.03 & \textbf{22.31} & 0.2813 & 0.4571 & \textbf{0.7407} \\
	    11.40\% (20)& 19.29 & 27.54 & \textbf{29.74} & 0.3845 & 0.8215 & \textbf{0.9451} \\
		16.37\% (30)& 21.14 & 32.29 & \textbf{34.92} & 0.4646 & 0.9595 & \textbf{0.9865} \\
		\hline
		\hline	
	\end{tabular}
	\caption{\label{table:brain} PSNR and SSIM values of Zero-filling, TV and our proposed method with different sampling rates on ``Brain'' image. The numbers in the brackets~(the first column) are the numbers of sampling spokes.}
\end{table}

\section{Conclusion}
In this paper, we propose a variational model solving linear inverse problems based on Wasserstein distance and total variation. The Benamou-Brenier energy is employed, and the given initial state image can be regarded as a template which can provide prior information. The theoretical analysis of the existence of solutions is discussed in measure space. Moreover, we employ the first-order primal-dual method to solve our problem and show some numerical experiments on undersampled MRI reconstruction. Such optimal transport based method can improve the reconstruction quality regarding the objective criterion and visual effect. In the future, we intend to extend this approach to multi-modality medical image reconstruction with unbalanced optimal transport energy, and to consider specific priors on momentum $m$ from physics.

\section*{Acknowledgments}
Yiming Gao is supported by Natural Science Foundation of Jiangsu Province (No. BK20220864).

\bibliographystyle{abbrv}
\bibliography{reference_abb}

\end{document}